\documentclass[dvi2pdf,epstopdf,12pt,a4paper,twoside]{amsart}

\usepackage{amsfonts}
\usepackage{amsmath}
\usepackage{amssymb}
\usepackage{amsthm}
\usepackage{caligr}
\usepackage{dsfont}
\usepackage{empheq}
\usepackage{enumerate}
\usepackage{enumitem}
\usepackage[multiple,symbol,stable]{footmisc}
\usepackage[hmargin={30mm,25mm},vmargin={30mm,30mm},includefoot]{geometry}
\usepackage{graphics} 
\usepackage{graphicx}
\usepackage[colorlinks]{hyperref}
\hypersetup{linkcolor=black, citecolor=black, urlcolor=black}
\usepackage{indentfirst}
\usepackage{latexsym}
\usepackage{parskip}
\usepackage{tikz}
\usetikzlibrary{positioning,fit,calc}
\usepackage[all]{xy}

\numberwithin{equation}{section}

\makeatletter
\renewcommand{\subsection}{\@startsection
{subsection}{2}{0mm}{\baselineskip}{-0.25cm}
{\normalfont\normalsize\em}}
\makeatother

\makeatletter
\def\gaps{\mathop{\operator@font Gaps}\nolimits}
\def\:={\mathrel{\mathop:}=}
\def\=:{=\mathrel{\mathop:}}
\makeatother

\def\neg1{\text{\boldmath$1$}}

\def\N{\mathds{N}}

\newtheorem{theorem}{Theorem}[section]
\newtheorem{proposition}[theorem]{Proposition}
\newtheorem{corollary}[theorem]{Corollary}
\newtheorem{lemma}[theorem]{Lemma}

{\theoremstyle{definition}
\newtheorem{definition}[theorem]{Definition}
\newtheorem{example}[theorem]{Example}

}

\theoremstyle{remark}

\makeatletter
\def\moverlay{\mathpalette\mov@rlay}
\def\mov@rlay#1#2{\leavevmode\vtop{%
   \baselineskip\z@skip \lineskiplimit-\maxdimen
   \ialign{\hfil$\m@th#1##$\hfil\cr#2\crcr}}}
\newcommand{\charfusion}[3][\mathord]{
    #1{\ifx#1\mathop\vphantom{#2}\fi
        \mathpalette\mov@rlay{#2\cr#3}
      }
    \ifx#1\mathop\expandafter\displaylimits\fi}
\makeatother

\makeatletter
\renewcommand*\env@matrix[1][*\c@MaxMatrixCols c]{%
  \hskip -\arraycolsep
  \let\@ifnextchar\new@ifnextchar
  \array{#1}}
\makeatother

\begin{document}
    \author[A. S. Castellanos]{A. S. Castellanos}

\author[G. Tizziotti]{G. Tizziotti}

    \title{On Weierstrass semigroup at $m$ points on curves of the type $f(y) = g(x)$}

\maketitle

\begin{abstract}
In this work we determine the so-called minimal generating set of the Weierstrass semigroup of certain $m$ points on curves $\mathcal{X}$ with plane model of the type $f(y) = g(x)$ over $\mathbb{F}_{q}$, where $f(T),g(T)\in \mathbb{F}_q[T]$. Our results were obtained using the concept of discrepancy, for given points $P$ and $Q$ on $\mathcal{X}$. This concept was introduced by Duursma and Park, in \cite{duursma}, and allows us to make a different and more general approach than that used to certain specific curves studied earlier.
\end{abstract}







\section{Introduction}

Let $\mathcal{X}$ be a nonsingular, projective, geometrically irreducible curve of genus $g \geq 1$ defined over a finite field $\mathbb{F}_q$, let $\mathbb{F}_q[\mathcal{X}]$ be the field of rational functions and $Div(\mathcal{X})$ be the set of divisors on $\mathcal{X}$. For $f \in \mathbb{F}_q[\mathcal{X}]$, the divisor of $f$ will be denoted by $(f)$ and the divisor of poles of $f$ by $(f)_{\infty}$. As follows, we denote $\mathbb{N}_{0} = \mathbb{N} \cup \{0\}$, where $\mathbb{N}$ is the set of positive integers. Let $P_{1}, \ldots , P_{m}$ be distinct rational points on $\mathcal{X}$. The set
$$
H(P_{1}, \ldots , P_{m}) = \{(a_{1}, \ldots, a_{m}) \in \mathbb{N}_{0}^ {m} \mbox{ ; } \exists f \in \mathbb{F}_q[\mathcal{X}] \mbox{ with } (f)_{\infty} = \sum_{i=1}^ {m} a_{i}P_{i} \}
$$
is called the \textit{Weierstrass semigroup} at the points $P_{1}, \ldots , P_{m}$. An element in $\mathbb{N}_{0}^{m} \setminus H(P_{1}, \ldots , P_{m})$ is called \textit{gap} and the set $G(P_{1}, \ldots , P_{m}) = \mathbb{N}_{0}^{m} \setminus H(P_{1}, \ldots , P_{m})$ is called \textit{gap set} of $P_{1}, \ldots , P_{m}$.

It is not difficult to see that the set $H(P_{1}, \ldots , P_{m})$ is a semigroup. The case $m=1$ has been studied for decades with its relationship to coding theory, see e.g. \cite{GarciaKimLax}, \cite{vanlint} and \cite{stichtenoth2}. An important fact about this case is that the cardinality of $G(P_1)$ is $g$. The case $m=2$ started to be studied by Kim, in \cite{kim}, where several properties were presented. Others relevant papers in the case $m = 2$ are \cite{ballico}, \cite{homma}, \cite{homma2} and \cite{tizziotti}. For $m>2$, this semigroup has been determined for some specific curves as Hermitian and Norm-trace curves, see \cite{gretchen1} and \cite{gretchen2}. With increasing interest in this semigroup, many results have been produced with several applications, especially in coding theory. Examples of these applications can be found in \cite{carvalho2}, \cite{GarciaKimLax} and \cite{gretchen}.

In this work we study the Weierstrass semigroup $H(P_{1}, \ldots , P_{m})$ for points on curves $\mathcal{X}$ with plane model of the type $f(y) = g(x)$ over $\mathbb{F}_{q}$, where $f(T),g(T)\in \mathbb{F}_q[T]$. Our results were obtained using the concept of discrepancy, for given points $P$ and $Q$ on $\mathcal{X}$, see Definition \ref{defi discrepancy}. This concept was introduced by Duursma and Park in \cite{duursma}, and it was our main tool for obtain the set $\Gamma(P_1,\ldots,P_m)$, called \textit{minimal generating} of $H(P_{1}, \ldots , P_{m})$, see Theorem \ref{maintheorem}. We observe that this approach is different from that used by Matthews in \cite{gretchen1} and Matthews and Peachey in \cite{gretchen2}.

This paper is organized as follows. Section 2 contains general results about Weierstrass semigroup and discrepancy. In Section 3, we determine the minimal generating for the Weierstrass semigroup $H(P_1,\ldots, P_m)$ for the curves $\mathcal{X}$ with plane model of the type $f(y) = g(x)$ cited above. Finally, in Section 4 we present examples for certain specific curves.

\section{Weierstrass semigroup and Discrepancy}

Let $\mathcal{X}$ be a non-singular, projective, irreducible, algebraic curve of genus $g \geq 1$ over a finite field $\mathbb{F}_{q}$.


Fix $m$ distinct rational points $P_1,\ldots,P_m$ on $\mathcal{X}$. Define a partial order $\preceq$ on $\N_0^m$ by $(n_1,\ldots,n_m)\preceq (p_1,\ldots,p_m)$ if and only if $n_i\leq p_i$ for all $i$, $1\leq i\leq m$.

For ${\bf u}_1,\ldots,{\bf u}_t\in \N_0^{m}$, where, for all $k$, ${\bf u}_k = (u_{k_{1}}, \ldots , u_{k_{m}})$, we define the \emph{least upper bound} ($lub$) of the vectors ${\bf u}_1,\ldots,{\bf u}_t$ in the following way: \[lub\{{\bf u}_1,\ldots,{\bf u}_t\}=(\max\{{ u_{1_1}},\ldots,{ u_{t_1}}\},\ldots, \max\{{ u_{1_m}},\ldots,{ u_{t_m}}\} )\in \N_0^{m}.\] For $\mathbf{n}=(n_1,\ldots,n_m)\in \N_0^{m}$ and $i \in \{ 1,\ldots , m\}$, we set
$$
\nabla_i (\mathbf{n}):=\{ (p_1, \ldots , p_m) \in H(P_{1}, \ldots, P_m) \mbox{ ; } p_i=n_i \}.
$$

\begin{proposition}\label{minimal}[\cite{gretchen1}, Proposition 3]
Let $\mathbf{n}=(n_1,\ldots,n_m)\in \N_0^{m}$. Then $\mathbf{n}$ is minimal, with respect to $\preceq$, in $\nabla_i(\mathbf{n})$ for some $i$, $1 \leq i \leq m$, if and only if $\mathbf{n}$ is minimal in $\nabla_i(\mathbf{n})$ for all $i$, $1 \leq i \leq m$.
\end{proposition}

\medskip

\begin{proposition}\label{lubH}[\cite{gretchen1}, Proposition 6]
Suppose that $1 \leq t \leq m \leq q$ and ${\bf u}_1,\ldots,{\bf u}_t\in H(P_{1}, \ldots , P_{m})$. Then $lub\{{\bf u}_1,\ldots,{\bf u}_t\} \in H(P_{1}, \ldots , P_{m})$.
\end{proposition}

\medskip

\begin{definition}
Let $\Gamma(P_{1})=H(P_1)$ and, for $m\geq 2$, define
$$
\Gamma(P_{1}, \ldots, P_{m}):=\{{\bf n}\in \N^m: \mbox{ for some } i, 1\leq i\leq m, {\bf n} \mbox{ is minimal in } \nabla_i ({\bf n})\}.
$$
\end{definition}

\medskip

\begin{lemma}\label{mPoints} [\cite{gretchen1}, Lemma 4]
For $m \geq 2$, $\Gamma(P_{1}, \ldots, P_{m}) \subseteq G(P_1)\times\cdots\times G(P_m)\;.$
\end{lemma}

\medskip

In \cite{gretchen1}, Theorem 7, it is shown that, if $2\leq m \leq q$, then $H(P_1,\ldots,P_m)= $

\[ \left\{\begin{array}{cl} lub\{{\bf u}_1,\ldots,{\bf u}_m\}\in \N_0^m: & {\bf u}_i \in \Gamma(P_{1}, \ldots, P_{m}) \\& \mbox{ or } ( u_{i_1}, \ldots, u_{i_k}) \in \Gamma(P_{i_1}, \ldots , P_{i_k}) \\ & \mbox{ for some } \{i_1,\ldots,i_k\}\subset\{1,\ldots,m\}  \mbox{ such that } \\ & i_1<\cdots<i_k \mbox{ and } u_{i_{k+1}} = \cdots = u_{i_\ell}=0, \\ & \mbox{ where }  \{i_{k+1},\ldots,i_m\}\subset\{1,\ldots,\ell\} \setminus \{i_1,\ldots,i_k\} \end{array}\right\}\;.\]

Therefore, the Weierstrass semigroup $H(P_1,\ldots,P_m)$ is completely determined by $\Gamma(P_1,\ldots,P_m)$. In \cite{gretchen1}, Matthews called the set $\Gamma(P_1,\ldots,P_m)$ of \emph{minimal generating} of $H(P_1,\ldots,P_m)$.

In \cite[Section 5]{duursma}, Duursma and Park introduced the concept of discrepancy as follows.

\begin{definition}\label{defi discrepancy}
A divisor $A \in Div(\mathcal{X})$ is called a \textit{discrepancy} for two rational points $P$ and $Q$ on $\mathcal{X}$ if $\mathcal{L}(A)\neq \mathcal{L}(A-P)=\mathcal{L}(A-P-Q)$ and $\mathcal{L}(A)\neq \mathcal{L}(A-Q)=\mathcal{L}(A-P-Q)$.
\end{definition}

The next result relates the concept of discrepancy with the set $\Gamma(P_1,\ldots,P_m)$.

\begin{lemma} \label{lemma discrepancy}
Let ${\bf n}=(n_1,\ldots,n_{m})\in H(P_1,\ldots,P_m)$. Then ${\bf n}\in \Gamma(P_1,\ldots,P_m)$ if and only if the divisor $A=n_1P_1+\cdots + n_m P_m$ is a discrepancy with respect to $P$ and $Q$ for any two rational points $P,Q\in \{P_1,\ldots,P_m\}$.
\end{lemma}

\begin{proof}
Let ${\bf n}=(n_1,\ldots,n_{m})\in \Gamma(P_1,\ldots,P_m)$, then there is a rational function $f\in \mathbb{F}_q(\mathcal{X})$ such that $(f)_\infty=A=n_1P_1+\cdots + n_m P_m$. Let $P_i,P_j\in \{P_1,\ldots,P_m\}$, with $i\neq j$. Suppose that $\mathcal{L}(A)=\mathcal{L}(A-P_i)$. Then we have that $f\in \mathcal{L}(A-P_i)$ and so $(f)+A-P_i\succeq 0$, contradiction because the pole divisor of $f$ is $A$. Therefore $\mathcal{L}(A) \neq\mathcal{L}(A-P_i)$. Now, suppose that $\mathcal{L}(A-P_i)\neq \mathcal{L}(A-P_i-P_j)$. Then there is a rational function $g\in \mathcal{L}(A-P_i)\setminus \mathcal{L}(A-P_i-P_j)$. So $(g)+A-P_i\succeq 0$ and this implies that $(g)_\infty= s_1 P_1 + \ldots + s_{j-1}P_{j-1} + n_j P_j + s_{j+1}P_{j+1} +  \ldots + s_{m} P_{m}$, with $s_i\leq n_i-1$ and $s_k\leq n_k$, for $k \in \{ 1, \ldots,m\} \setminus \{i,j\}$. It follows that $(s_1,\ldots,s_{j-1}, n_j, s_{j+1}, \ldots,s_m)\in \triangledown_j({\bf n})$, this contradicts the fact that ${\bf n}\in \Gamma(P_1,\ldots,P_m)$ is minimal in $\triangledown_j({\bf n})$. Similar considerations apply when initializing with  $P_j$. Therefore, $A=n_1P_1+\cdots + n_m P_m$ is a discrepancy with respect to $P$ and $Q$ for any two rational points $P,Q\in \{P_1,\ldots,P_m\}$. Conversely, suppose that ${\bf n}=(n_1,\ldots,n_{m})\in H(P_1,\ldots,P_m)$ and that the divisor $A=n_1P_1+\cdots + n_m P_m$ is a discrepancy with respect to any $P_i,P_j\in \{P_1,\ldots, P_m\}$, with $i\neq j$. Since ${\bf n}=(n_1,\ldots,n_{m})\in H(P_1,\ldots,P_m)$, there is a rational function $f\in \mathbb{F}_q(\mathcal{X})$ such that $(f)_\infty=n_1P_1+\cdots + n_m P_m$. Suppose that ${\bf n}=(n_1,\ldots,n_m)$ is not minimal in $\triangledown_i({\bf n})$ for some $i \in \{1,\ldots,m\}$. Then, there exist a rational function $f_i \in \mathbb{F}_q(\mathcal{X})$ such that $(f)_\infty=s_1P_1+\cdots+s_m P_m$ with ${\bf s}=(s_1,\ldots,s_m)\in \triangledown_i({\bf n})$, ${\bf s}\neq {\bf n}$ and ${\bf s}\preceq {\bf n}$. Therefore, $s_i=n_i$ and $s_j<n_j$ for some $j\neq i$. It follows that $f_i\in \mathcal{L}(A-P_j)$ and $f_i\not\in \mathcal{L}(A-P_i)$, this contradicts the fact that $\mathcal{L}(A-P_j)=\mathcal{L}(A-P_i)$.
\end{proof}

\section{Weierstrass semigroup $H(P_1,\ldots,P_m)$ for certain types of curves} Consider a curve $\mathcal{X}$ over $\mathbb{F}_{q}$ given by affine equation $f(y) = g(x)$ ,where $f(T), g(T)\in \mathbb{F}_q[T]$, $\deg(f(y))=a$ and $\deg(g(x))=b$, with $gdc(a,b)=1$. Suppose that $\mathcal{X}$ is absolutely irreducible with genus $g=(a-1)(b-1)/2$.

Let $P_1, P_2, \ldots, P_{a+1}$ be $a+1$ distinct rational points such that

\begin{equation}\label{eq1}
a P_1 \sim P_2 + \cdots + P_{a+1}\,,
\end{equation}

and

\begin{equation}\label{eq2}
b P_i \sim b P_j \mbox{, for all } i,j\in \{1,2,\ldots,a+1\},
\end{equation}
where \lq\lq $\sim$" represent the equivalence of divisors. Note that $H(P_1)= \langle a, b \rangle$.

Examples of such curves can be found in \cite{kummer}, \cite{kondo}, \cite{gretchen1} and \cite{gretchen2}.

Let $1 \leq m \leq a+1 \leq q$. For

\begin{equation}\label{eq3}
t + \sum_{j=2}^m s_j = a+1-m\;,
\end{equation}

we have that the equivalences (\ref{eq1}) and (\ref{eq2}) yield

\begin{equation}\label{eq4}
(tb-ia) P_1 + \sum_{j=2}^m (s_j b + i) P_j \sim \sum_{j=m+1}^{a+1} (b-i) P_j\,.
\end{equation}

Observe that all coefficients are positive for
\begin{equation}\label{eq5}
0 < ia < tb, \qquad  s_j \geq 0\;.
\end{equation}

Note that $0<t \leq a$ and thus $0<i < b$. 

From the basic equivalence (\ref{eq1}) and (\ref{eq2}) we have that
\begin{equation}\label{eq6}
A=(b-i)(aP_1-P_2-\cdots -P_m)\sim\sum_{j=m+1}^{a+1} (b-i) P_j\;.
\end{equation}

By (\ref{eq2}) the divisor on the left can be replaced with an efficient equivalent divisor of the follow form
\begin{equation}\label{eq7}
A'=((a+1-m)b-ia)P_1+i(P_2+\cdots+P_m)\sim \sum_{j=m+1}^{a+1} (b-i) P_j\;.
\end{equation}

Thus, by redistributing $a+1-m=s+t$ over $P_1$ and $P_2$, the divisors in (\ref{eq4}), with $s=\sum_{j=2}^m s_j$, contain the special representative
\begin{equation}\label{eq8}
(tb-ia)P_1+(sb+i)P_2+i(P_3+\cdots+P_m)\sim \sum_{j=m+1}^{a+1} (b-i) P_j\;.
\end{equation}

The other divisors with same $t,s$ and $i$ are easy obtained from (\ref{eq8}) by distributing $s$ in all possible ways over $P_2,\ldots,P_{m+1}$ such that $s=\sum_{j=2}^m s_j$.


\begin{proposition} \label{prop discrepancy}
Let $b$ and $i$ be as above. Then, the divisor $(b-i)(aP_1-P_2-\cdots-P_m)$ is a discrepancy with respect to $P$ and $Q$ for any two distinct points $P,Q\in\{P_1,\ldots,P_m\}$.
\end{proposition}
\begin{proof}
It suffices to prove the claim for the equivalent but effective divisor
$A' = ((a+1-m)b-ia) P_1 + i (P_2 + \cdots + P_m).$ The divisors $A$ and $A'$ appear on the left side of the
divisor equivalences (\ref{eq6}) and (\ref{eq7}), respectively. The equivalence of effective divisors in (\ref{eq7}) gives a rational
function $f \in L(A')$ with pole divisor equal to $A'$. Thus $L(A') \neq L(A'-P)$ for any point $P$.
To prove that $L(A'-P) = L(A'-P-Q)$ we consider the equivalent statement $L(K+P+Q-A') \neq L(K+P-A')$.
For the choice of canonical divisor $K=(ab-a-b-1)P_1$,
$$
\begin{array}{cl}
K+P+Q-A' &= (ab-a-b-1)P_1 + P + Q - (a+1-m)b-ia) P_1\\
                 &  - i (P_2 + \cdots + P_m) \\
        &= (i-1)a+(m-2)b-1)P_1 + P + Q - i (P_2 + \cdots + P_m).
\end{array}
$$

Consider first the case $P_1 \in \{ P, Q \}$. Without loss of generality we may assume that $P=P_1,Q=P_2$.
Let $f_2, \ldots, f_m$ be the functions with divisors
\[
(f_j) = \left\{ \begin{array}{ll} -a P_1 + (P_2 + \cdots + P_{a+1}), & j=2 \\
                      b ( P_j - P_1), & j > 2 \end{array} \right .
\]

Then $f_2^{i-1} f_3 \cdots f_m \in L(K+P+Q-A') \backslash L(K+Q-A').$ Thus $L(A'-Q) = L(A'-P-Q)$.
From $L(A') \neq L(A'-Q) = L(A'-P-Q)$ and $L(A') \neq L(A'-P)$ it follows that
also $L(A'-P) = L(A'-P-Q).$
Consider next the case $P_1 \not \in \{ P, Q \}$, say $P=P_2, Q=P_3.$ Thus, we have that $f_2^{i-1} f_4 \cdots f_m \in L(K+P+Q-A') \backslash L(K+Q-A').$
As before, it follows that $L(A'-Q) = L(A'-P-Q)$ and that $L(A'-P) = L(A'-P-Q).$
\end{proof}

\begin{corollary} \label{corolary discrepancy}
Let $a$, $b$, $t$, $i$, $s_2, \ldots , s_m$ be as above. Then, the divisor $(tb-ia) P_1 + \sum_{j=2}^m (s_j b + i) P_j$ is a discrepancy with respect to $P$ and $Q$ for any two distinct points $P,Q\in\{P_1,\ldots,P_m\}$.
\end{corollary}
\begin{proof}
Follows directly from the previous proposition and equations (\ref{eq4}) and (\ref{eq8}).
\end{proof}

\begin{theorem}\label{maintheorem}
Let $\mathcal{X}$ and $P_1, P_2, \ldots, P_{a+1}$ be as above. For $2 \leq m \leq a+1$, let
$$
\displaystyle S_m=\left\{  (tb-ia, s_2b+i, \ldots , s_m b + i) \mbox{; } t+\sum_{j=2}^{m}s_j=a+1-m \mbox{, } 0<ia<tb \mbox{, } s_j\geq 0 \right\}.
$$
Then, $\Gamma(P_1,\ldots,P_{m}) = S_m$.
\end{theorem}
\begin{proof}
By Corollary \ref{corolary discrepancy}, the divisor $(tb-ia) P_1 + \sum_{j=2}^m (s_j b + i) P_j$ is a discrepancy with respect to $P$ and $Q$ for any two distinct points $P,Q\in\{P_1,\ldots,P_m\}$. So, by Lemma \ref{lemma discrepancy}, follows that $S_m \subseteq \Gamma(P_1,\ldots,P_{m})$.

Next, we show that $\Gamma(P_1,\ldots,P_{m})\subseteq S_m$. Let $\mathbf{n} = (n_1 , \ldots , n_m)\in \Gamma(P_1,\ldots,P_{m})$. By Lemma \ref{mPoints},
$\mathbf{n} = (n_1, \ldots, n_m)\in G(P_1)\times G(P_2)\times\cdots\times G(P_m)$.

As $H(P_1)=\langle a,b \rangle$, from Lemma 1 in \cite{rosales} we have that $n_1=ab - i_1a- j_1b = (a - j_1)b - i_1 a$, for some $i_1,j_1 \in \mathbb{N}$. Let $\lambda=a-j_1$. Note that $0<j_1<a$ and $0<ai_1 < b\lambda$.

By Equation (\ref{eq2}), follows that $b \in H(P_{\ell})$, for $2 \leq \ell \leq m$. So, we have that $n_{\ell} = s_{\ell} b + i_{\ell}$, where $0 < i_{\ell} < b$ and $s_{\ell} \geq 0$.



Let $i=\min\{i_\ell: 2\leq \ell\leq m\}$. By Equation (\ref{eq1}), for each $\ell=2,\ldots,m$, there is a rational function $h_{\ell}$ such that $(h_{\ell}) = bP_1 - bP_{\ell}$. By Equation (\ref{eq1}), there is a rational function $g$ such that $(g)=P_2 + \cdots + P_{a+1} - a P_1$.

Let $\omega=g^{b-i} . \prod_{\ell=2}^{m} h_{\ell}^{s_{\ell} + 1}$. Then, $(\omega)_{\infty} = ((a-\sum_{\ell=2}^m (s_{\ell}+1))b-ia)P_1+ (s_2b+i)P_2 + \cdots + (s_mb+i) P_m$. Taking $t=a-\sum_{\ell=2}^m (s_{\ell}+1)$, by Corollary \ref{corolary discrepancy}, $(\omega)_{\infty}$ is a discrepancy with respect to $P$ and $Q$ for any two distinct points $P,Q\in\{P_1,\ldots,P_m\}$. So, by Lemma \ref{lemma discrepancy}, $\mathbf{w}=(tb-ia, s_2b+i, \ldots, s_mb+i) \in \Gamma(P_1, \ldots, P_m)$. Now, we know that $i=i_k$, for some $k \in \{ 2,\ldots,m\}$. Then, $\mathbf{w} \in \nabla_k ({\bf n})$ and by minimality of $\mathbf{w}$ and $\mathbf{n}$ follows that $\mathbf{w}=\mathbf{n}$ and $\Gamma(P_1,\ldots,P_{m}) \subseteq S_m$.







\end{proof}

\section{Examples}

\begin{example}
Let $\mathcal{X}_{q^{2r}}$ be the curve defined over $\mathbb{F}_{q^{2r}}$ by the affine equation
$$
y^q + y = x^{q^r + 1} ,
$$
where $r$ is an odd positive integer and $q$ is a prime power. Note that $\mathcal{X}_{q^{2}}$ is just the Hermitian curve. The curve $\mathcal{X}_{q^{2r}}$ has genus
$g=q^{r}(q-1)/2$, a single point at infinity, namely $P_1=(0:1:0)$, and others $q^{2r+1}$ rational points. It is important to observe that $\mathcal{X}_{q^{2r}}$ is a quotient of the Hermitian curve and thus is a maximal curve over $\mathbb{F}_{q^{2r}}$. The Weierstrass semigroup $H(P_1,P_2)$ was studied in \cite{st} and more details about this curve can be found in \cite{kondo}.

Let $y_{1}, \ldots, y_{q}$ be the solutions in $\mathbb{F}_{q^{2r}}$ to $y^q + y = 0$. Let $P_2=(0,y_1), P_{3}=(0,y_2), \ldots , P_{q+1}=(0,y_q)$. Since $(x)=P_2+\ldots+P_{q+1} - qP_1$ and $(y-y_j)=(q^{r}+1)P_{j+1} - (q^r + 1)P_1$, for all $j = 1,\ldots,q$, we have that

\begin{equation}\label{eq1 div kondo}
q P_1 \sim P_2 + \cdots + P_{q+1}\,,
\end{equation}

and

\begin{equation}\label{eq2 div kondo}
(q^r + 1) P_i \sim (q^r + 1) P_j \mbox{, for all } i,j\in \{1,2,\ldots,q+1\},
\end{equation}

For this curve, take $m=4$, $q=5$ and $r=3$ we have that $a=5$ and $b=126$. Then by Theorem (3.3) follows that $\Gamma(P_1,P_2,P_3,P_4)$ consists of the following 125 elements
$$
\begin{array}{llll}
(126,0,0,126)& + & i(-5,1,1,1)\;, & i=1,\ldots, 25\;, \\
(126,0,126,0)& + & i(-5,1,1,1)\;, & i=1,\ldots, 25\;, \\
(126,126,0,0)& + & i(-5,1,1,1)\;, & i=1,\ldots, 25\;, \\
(252,0,0,0) & + & i(-5,1,1,1)\;,  & i=1,\ldots, 50\;.
\end{array}
$$

\end{example}

\begin{example}
Let a Kummer extensions given by $y^b=g(x)=\prod\limits_{i=1}^{a}(x-\alpha_i)$ where
 $g(x)$ is a separable polynomial over $\mathbb{F}_q$ of degree $a$  and  $\gcd(a, b)=1$. We know that these curves has a single point $P_1$ at infinity and has genus $(b-1)(a-1)/2$, see \cite{kummer}. Then we have the following divisors in $F(\mathcal{X})$:

\begin{enumerate}
\item $(x-\alpha_i)=bP_i-bP_1$  for every $i$, $2\le i \le a+1$,
\item $(y)= P_2+\cdots + P_{a+1}-a P_1$,
\end{enumerate}
\end{example}

For $a=5$ and $b=7$ we have that

$$
\begin{array}{lcl}
\Gamma(P_1,P_2)& =& \{(23,1),(18,2),(13,3),(8,4),(3,5),(16,8),\\ & &(11,9),(6,10),(1,11),(9,15),(4,16),(2,22)\}\;.\\
& &\\
\Gamma(P_1,P_2,P_3)& =& \{(2,8,8),(2,15,1),(2,0,15),(9,8,1),(9,1,8),\\ & &(4,9,2),(4,2,9),(16,1,1),(11,2,2),(6,3,3),(1,4,4)\}\;.\\
& &\\
\Gamma(P_1,P_2,P_3,P_4)& =& \{(2,8,1,1),(2,1,8,1),(2,1,1,8),(9,1,1,1),(4,2,2,2)\}\;.\\
& &\\
\Gamma(P_1,P_2,P_3,P_4,P_5)& =& \{(2,1,1,1,1)\}\;.
\end{array}
$$

\section{Acknowledgment}

The authors would like to thank I. Duursma for very useful suggestions and comments that improved the results and the presentation of this work.

\end{document}